\documentclass[12pt,a4paper]{article}
\usepackage[active]{srcltx}
\usepackage{amsfonts}
\usepackage{amsmath}
\usepackage{amsbsy}
\usepackage{amsxtra}
\usepackage{latexsym}
\usepackage{amssymb}
\usepackage{url}
\usepackage{pstricks,pst-node}
\usepackage{cite}
\makeatletter
\g@addto@macro\thesection.
\makeatother

\newtheorem{theorem}{Theorem}

\newtheorem{lemma}{Lemma}
\newtheorem{corollary}{Corollary}
\newtheorem{proposition}{Proposition}
\newtheorem{remark}{Remark}
\newtheorem{definition}{Definition}
\newtheorem{example}{Example}

\DeclareMathOperator{\CP}{CP}
\DeclareMathOperator{\LCP}{LCP}

\DeclareMathOperator{\SOL}{SOL}
\DeclareMathOperator{\cone}{cone}
\DeclareMathOperator{\intr}{int}

\DeclareMathOperator{\glm}{GL(m,\mathbb R)}

\DeclareMathOperator{\dy}{\mathcal O}

\newcommand{\lng}{\langle}
\newcommand{\rng}{\rangle}

\newcommand{\R}{\mathbb R}

\newcommand{\tp}{^\top}

\newcommand{\pb}{\textbf{P}}
\newcommand{\pq}{\textbf{Q}}
\newcommand{\pf}{\textbf{FS}}

\newenvironment{proof}{{\noindent\bf Proof.}}{\hfill$\Box$\\}

\begin{document}

\large

\title{Linear complementarity on simplicial cones and the congruence orbit of matrices
}
\author{A. B. N\'emeth\\Faculty of Mathematics and Computer Science\\Babe\c s Bolyai University, Str. Kog\u alniceanu nr. 1-3\\RO-400084
Cluj-Napoca, Romania\\email: nemab@math.ubbcluj.ro \and S. Z. N\'emeth\\School of Mathematics, University of Birmingham\\Watson Building,
Edgbaston\\Birmingham B15 2TT, United Kingdom\\email: s.nemeth@bham.ac.uk}
\date{}
\maketitle

\begin{abstract}
The congruence orbit of a matrix has a natural connection
with the linear complementarity problem on simplicial cones formulated for the
matrix. In terms of the two approaches -- the congruence orbit and the family of
all simplicial cones -- we give equivalent classification of matrices from 
the point of view of the complementarity theory.
\end{abstract}

\section{Introduction}

We use in this introduction some standard terms and notations which will also be
specified in the next section.

Let be $K\subset \R^m$ a cone, $K^*\in \R^m$ be its dual, 
$M:\R^m\to \R^m$ a linear mapping and $q\in \R^m$.
The problem
\begin{equation*}
	\LCP(K,q,M): \;\textrm{find}\;x\in K \;\textrm{with}\; Mx+q\in K^*\;\textrm{and}\; \lng x,Mx+q\rng=0
\end{equation*}
is called \emph{the linear complementarity problem on the cone $K$}.
In the case $K=\R^m_+$ it is denoted by $\LCP(q,M)$ and called
\emph{the classical linear complementarity problem}.

As one of the most important problems in optimization theory, 
the classical linear complementarity problem has a broad literature (see \cite{MR3396730}
and the literature therein).
Despite of the important progress last decades in this field, it still
in the center of interest nowadays.

Besides the classical case, the linear complementarity on
Lorentz cone and the cone of positive semi-definite matrices emerged as an important
topic in the previous decade \cite{MR1782157,GowdaSznajderTao2004}.

When $K\subset \R^m$ is a simplicial cone, the linear complementarity
problem can be transformed by a linear mapping in the classical one. But 
general simplicial cones can differ substantially from each other in some aspects, 
one of them being the projection onto the cone, the mapping playing 
an essential role in the solution of optimization problems.  
If the linear mapping $M$ is given, such an approach relates the problem to
the \emph{congruence orbit} of $M$, i.e. to the set of maps
\begin{equation}\label{orbit}
\dy (M) =\{L\tp ML:\;L\in \glm \},
\end{equation}
where $\glm$ denotes the \emph{general linear group} of $\R^m$,
i.e., the group of invertible linear maps of the vector space $\R^m$.

Among other results, in this note we will show 
that $\LCP(K,q,M)$ is feasible
for an arbitrary simplicial cone $K\subset \R^m$ and an arbitrary $q\in \R^m$
if and only if 
$M$ is a \emph{positive definite} mapping \cite{MR3396730}, i.e., if 
$\lng Mx,x \rng > 0$, $\forall x\in \R^m$, $x\ne 0$,
which is equivalent to saying that this property holds 
for each member of $\dy (M)$. It turns out that this
property is also equivalent with the much stronger $\pb$ and $\pq$-properties of all members of $\dy (M)$ and equivalently,
with the corresponding properties of $M$ for each simplicial cone $K$.

It is possible that some of the problems considered in the present note 
already occured in a different setting in the huge literature on
linear complementarity. Even so, the approach of considering
classical $\pb$ and $\pq$-properties of a matrix for the \emph{whole family}
of simplicial cones and the relation with the congruence orbit of the matrix
seems a novel approach which justifies our following investigation.

\section{Terminology and notations} \label{terminology}

Denote by $\R^m$ the $m$-dimensional Euclidean space endowed with the scalar 
product $\lng\cdot,\cdot\rng:\R^m\times\R^m\to\R,$ and the Euclidean norm $\|\cdot\|$ and topology 
this scalar product defines. Denote $\lng x,y\rng=0$ by $x\perp y$.

Throughout this note we shall use some standard terms and results from convex geometry 
(see e.g. \cite{MR1451876}). 

Let $K$ be a \emph{convex cone} in $\R^m$, i. e., a nonempty set with
(i) $K+K\subset K$ and (ii) $tK\subset K,\;\forall \;t\in \R_+ =[0,+\infty)$.
The convex cone $K$ is called \emph{pointed}, if $K\cap (-K)=\{0\}.$
The cone $K$ is {\it generating} if $K-K=\R^m$. $K$ is generating
if and only if $\intr K\not= \varnothing.$ 

A closed, pointed generating convex cone is called \emph{proper cone}. 
 
The set
\begin{equation}\label{simplicial}
 K= \cone\{x^1,\dots,x^m\}:=\{t_1x^1+\dots+t_m x^m:\;t_i\in \R_+,\;i=1,\dots,m\}
\end{equation}
with $x^1,\,\dots,\,x^m$ linearly independent vectors in $\R^m$ is called a \emph{simplicial cone}.
A simplicial cone is proper.

The \emph{dual} of the convex cone $K$ is the set
$$K^*:=\{y\in \R^m:\;\lng x,y\rng \geq 0,\;\forall \;x\in K\},$$
with $\lng\cdot,\cdot\rng $ is the standard scalar product in $\R^m$.

The cone $K$ is called \emph{self-dual} if $K=K^*.$ If $K$
is self-dual, then it is proper.

In all that follows we will suppose that $\R^m$ is endowed with a
Cartesian system having an orthonormal basis $e^1,\dots,e^m$
and the elements $x\in \R^m$ are represented by the column vectors
$x=(x_1,\dots,x_m)\tp$, with $x_i$ the coordinates of $x$
with respect to this basis. (That is, $\R^m$ will be the vector space of $m$-dimensional column vectors.) 

The set
\[\R^m_+=\{x=(x_1,\dots,x_m)\tp\in \R^m:x_i\geq 0,\mbox{ }
i=1,\dots,m\}\]
is called the \emph{non-negative orthant} of the above introduced Cartesian
reference system.
In fact
$$\R^m_+=\cone \{e^1,...,e^m\}.$$
 A direct verification shows that $\R^m_+$ is a
self-dual cone. 

If $K$ is the simplicial cone defined by (\ref{simplicial}) and $A$ is the non-singular matrix
transforming the basis $e^1,...,e^m$ to the linear independent vectors $x^1,...,x^m$,
then obviously
\begin{equation}\label{repre}
K=A\R^m_+.
\end{equation}

For simplicity from now on we will call a convex cone simply cone.

\section{Changing the cone linearly}

\begin{lemma}\label{gentran}
Let $W\subset \R^m$ be a cone and $A\in \glm$. Then $K=AW$ is a cone too
and $K^*=A^{-T}W^*$.
\end{lemma}

\begin{proof}
The first assertion is trivial.

Take $y\in K^*$. This is equivalent to 
$$\lng Aw,y\rng =\lng w,A\tp y\rng \geq 0,\mbox{ }\forall w\in W.$$
Thus, $y\in K^*\iff A\tp y\in W^* \iff y\in A^{-T}W^*.$
\end{proof}

\begin{corollary}\label{gentranself}
If $W^*=W$, $A\in\glm$ then
$$ K=AW \iff K^*=A^{-T}W.$$
If $K$ is the simplicial cone (\ref{simplicial}), then,
because of the representation (\ref{repre}) and the self-duality
of $\R^m_+$, we have
$$K^*=A^{-T}\R^m_+.$$

\end{corollary}

\section{Linear transformation of a cone and the complementarity problem}

For the mapping $F:\R^m \to \R^m$ the \emph{complementarity problem} $\CP(F,K)$ is to find $x\in\R^m$ such that

\begin{equation*}
	K\ni x\perp F(x)\in K^*.
\end{equation*}

The solution set of $\CP(F,K)$ will be denoted by $\SOL(F,K)$. We have 
\begin{equation}\label{-T}
x\perp y\iff Ax\perp A^{-\top}y. 
\end{equation}
Hence, by using Lemma \ref{gentran} and 
(\ref{-T}), we conclude the following result:

\begin{proposition}\label{CPequ}
If $W$ is a cone, $A\in \glm$ and $K=AW$, then
$$\SOL(F,K)=A(\SOL(A\tp FA,W)).$$
\end{proposition}

\begin{proof}
	Indeed,
	\begin{gather*}
		Ax\in A(\SOL(A\tp FA,W))\iff x\in\SOL(A\tp FA,W)\iff W\ni x\perp A\tp F(Ax)\in W^*\\\iff K\ni Ax\perp F(Ax)\in K^*
		\iff Ax\in SOL(F,K). 
	\end{gather*}
\end{proof}

\section{The case of linear complementarity}

The complementarity problem $\CP(f,K)$ with $F(x)=Mx+q$, where $M\in \R^{m\times m}$ and $q\in \R$,
will be denoted by $\LCP(K,M,q)$ and called \emph{linear complementarity problem}. Thus, the linear complementarity problem $\LCP(K,M,q)$ is to
find $x\in\R^m$ such that 

\begin{equation*}
	K\ni x\perp Mx+q\in K^*.
\end{equation*}

The solution set of $\LCP(K,M,q)$ will be denoted by $\SOL(K,M,q)$. In this case Proposition \ref{CPequ} becomes

\begin{proposition}\label{LCPequ}
If $W$ is a cone, $A\in \glm$ and $K=AW$, then
$$\SOL(K, M,q)=A(\SOL(W,A\tp MA, A\tp q)).$$
\end{proposition}
	
If $K=\R^m_+$, then $\CP(F,K)$, $\SOL(F,K)$, $\LCP(K,M,q)$ and $\SOL(K,M,q)$ will simply be denoted by 
$\CP(F)$, $\SOL(F)$, $\LCP(M,q)$ and $\SOL(M,q)$, respectively.


\section{The congruence orbit of a matrix and the complementarity problem}

If $A$ and $B$ are in $\R^{m\times m}$, then they are \emph{congruent} and
we write $A\sim B$, if there exists $L\in \glm$
such that
$$L\tp AL=B,$$
that is $B$ is in the congruence orbit $\dy (A)$ of $A$ defined at (\ref{orbit}).
Obviously, $\sim$ is an equivalence relation 
and in this case $\dy (B)=\dy (A).$


In the case of simplicial cones Proposition \ref{LCPequ} reduces to

\begin{proposition}\label{LCPclassic}
If $K=L\R^m_+$ is a simplicial cone then
$$\SOL(K, A,q)=A(\SOL(\R^m_+,L\tp AL,A\tp q)).$$
Hence the linear complementarity problem on a simplicial cone is equivalent
to the complementarity problem on the non-megative orthant, that is,
to the classical linear complementarity problem. 
\end{proposition}

\begin{remark}
Proposition \ref{LCPclassic} shows that for linear complementarity
problems with matrix $A$ on simplicial cones the congruence
orbit $\dy (A)$ of $A$ appears in a natural way.
\end{remark}

\begin{definition}
	Let $A$ be a linear transformation. Then, we say that
	\begin{enumerate}
\item $A$ has the \emph{$K$-\pq-property} if $\LCP(K,A,q)$ has a solution for all $q$.
\item $A$ has the \emph{$K$-\pb-property} if $\LCP(K,A,q)$ has a unique solution for all $q$.
\item The $\R^m_+$-\pq-property ($\R^m_+$-\pq-property) is called \emph{\pq-property, (\pb-property)} and the matrix of the linear transformation
	$A$ with the \pq-property (\pb-property) is called \pq-matrix (\pb-matrix) (for simplicity we denote a linear transformation
	and its matrix by the same letter).
\item $A$ has the \emph{general feasibility property} with respect to $K$ denoted $K$-\pf-property 
	if \[(AK+q)\cap K^*\ne\varnothing,\mbox{ }\forall q\in \R^m.\]
	If $K=\R^m_+$, then the $K$-\pf-property is called \pf-property and it is characterized by the relation
	\[(A\R^m_++q)\cap\R^m_+\ne\varnothing,\mbox{ }\forall q\in \R^m.\] The matrix $A$ with the \pf-property is called \pf-matrix.
 \end{enumerate}
\end{definition}

\begin{remark}\label{pbpos}
Obviously, the \pb-property of a matrix $A$ implies its \pq-property, and its \pq-property implies its \pf-property as well.
A classical result going back to the paper \cite{Samelson} asserts that
$A$ possesses the \pb-property if and only if all its principal minors are positive.
Theorem 3.1.6 in \cite{MR3396730} asserts that a positive definite matrix possesses the $\pb$-property.

The \pf-property of a matrix can be considered the weakest one in the context
of linear complementarity. It is easy to see that the \pf-property ($K$-\pf-property) of the matrix $A$
is equivalent to $-A\R^m_++\R^m_+=\R^m$ ($-AK+K^*=\R^m$). 
\end{remark}

With the notations in the above definition we have

\begin{proposition}\label{ekviv}
If $A\in \R^{m\times m}$ has the $K$-\pq-property ($K$-\pb-property) then $M=L\tp AL\in \dy (A)$
has the $LK$-\pq-property ($LK$-\pb-property).
\end{proposition}

\begin{example}\label{Icongr}

\emph{The congruence orbit of the identity}

We have $\dy (I)=\{L\tp L:\;L\in\glm\}$.

Hence, each member of $\dy (I)$ is a \emph{symmetric positive definite matrix}.
\end{example}

The following lemma is based on Example \ref{Icongr} and shows that if the congruence orbit of a matrix contains a positive definite matrix, then all of its matrices are 
positive definite. 

\begin{lemma}\label{ulem}
If $\dy (A)$ contains a symmetric positive definite matrix, then
$$\dy (A)= \dy (I).$$
\end{lemma}
\begin{proof}
We can suppose that $A$ itself is a symmetric positive definite matrix.

If we denote by $R$ the square root of $A$ \cite{MR2978290}, then we can write
\begin{equation*}
\dy (A) =\{L\tp AL:\;L\in \glm\} = \{(RL)\tp (RL):\;L\in \glm\}= 
\end{equation*}
\begin{equation*}
\{M\tp M:\;M\in\glm\}=\dy (I).
\end{equation*}
\end{proof}

Obviously, a symmetric positive definite matrix is nothing else but a symmetric \pb-matrix.
Hence, by Lemma \ref{ulem} we conclude

\begin{corollary}\label{corulem}
Each member of the congruence orbit of a symmetric positive definite matrix is a \pb-matrix.
\end{corollary}

How about the congruence orbit of a non-symmetric \pb-matrix? Can it have a property similar to the one
stated in Corollary \ref{corulem}? We will show that this holds if and only if the matrix is
positive definite.

\begin{lemma}\label{dgnl}
If the diagonal of the matrix $A\in \R^{m\times m}$ contains some non-positive element, then
$\dy (A)$ contains non-\pf-matrices.
\end{lemma}

\begin{proof}
(a) Let $D_k$ be a diagonal matrix with $d_{kk}=-1$ and $d_{ii}=1$ if $i\ne k$.
If $A=(a_{ij})_{i,j=1,...,m}$, then
$B=D_k\tp AD_k$ is a matrix with $b_{ij}=a_{ij}$ if $i\ne k$ and $j\ne k$, $b_{ki}=-a_{ki}$, $i\ne k$, $b_{jk}=-a_{jk}$, $j\ne k$, and
$b_{kk}=a_{kk}$.

(b) Without loss of generality we can assume that $a_{11}\leq 0$. Assume that the positive terms of the first line of $A$
are $a_{1j},\,a_{1k},,...,a_{1l}$. Then if $D$ is the diagonal matrix with $d_{ii}=-1,\;i \in \{j,k,...,l\}$
and $d_{ii}=1,\;i\notin \{j,k,...,l\}$, then using the remark at (a) we conclude that
$$C=D\tp AD$$
is a matrix whose first line contains only non-positive elements. Hence, for any $q=(-1,*....*)\tp$ we have
$(C\R^m_++q)\cap \R^m_+=\varnothing$. 
\end{proof}

As we have stated at Remark \ref{pbpos} a positive definite matrix possesses the \pb-property, but simple
examples show, that the converse is not true.

We have the obvious assertion:

\begin{lemma}\label{szimmertiz}
If $A\in \R^{m\times m}$ is not positive definite, then its \emph{symmetrizant}
\begin{equation}\label{szimm}
S(A)=\frac{A+A\tp}{2}
\end{equation}
is not positive definite neither.
\end{lemma}

\begin{remark}\label{szimmdiag}
	Observe that the diagonal of $S(A)$ defined as in (\ref{szimm}) coincides with the diagonal of $A$.
\end{remark}

\begin{theorem}\label{eqn}
Let $A\in \R^{m\times m}$. Then the following assertions are equivalent:
\begin{enumerate}
\item Each member of $\dy (A)$ is a \pf-matrix.
\item Each member of $\dy (A)$ is a \pq-matrix.
\item  Each member of $\dy (A)$ is a \pb-matrix.
\item $A$ is a positive definite matrix.
\end{enumerate}
Hence, if any of the above conditions hold then
$A$ possesses the $K$-\pb, $K$-\pq, $K$-\pf- properties
for any simplicial cone $K$.
\end{theorem}

\begin{proof}
Suppose that assertion 1 hold.

(a) Assume that $A$ is not positive definite, that is, item 4. does not hold.
Then, by Lemma \ref{szimmertiz}, the same is true for $S(A)$. That is, $S(A)$ is
a symmetric matrix which is not positive definite. Then, it has non-positive
eigenvalues, that is, in the spectral decomposition
$$S(A)=ODO\tp$$
$D$ is a diagonal matrix with some non-positive elements.
On the other hand we have
\begin{equation*}
	D= O\tp S(A)O=\frac{O\tp AO +O\tp A\tp O}{2}.
\end{equation*}
Now, since $O\tp A\tp O=(O\tp AO)\tp$, it follows, according to
Remark \ref{szimmdiag}, that the diagonal of $D$ coincides with the diagonal of
$O\tp AO$. Since this diagonal contains non-positive elements, it follows
by Lemma \ref{dgnl} that the orbit $\dy (A)$ contains
elements which are not \pf-matrices. This shows the
implication $1\Rightarrow 4$.

(b) If 4 holds, then for each $L\in \glm$ and any $x\not= 0$ we have
$$\lng L\tp AL x,x\rng =\lng ALx,Lx\rng >0,$$
that is, $L\tp AL$ is positive definite and hence, by Theorem 3.1.6 in \cite{MR3396730}
is a \pb-matrix, and hence also a \pq-matrix and an \pf-matrix.
Thus, we have the implications $4 \Rightarrow 3 \Rightarrow 2 \Rightarrow 1$.

(c) The  last assertion of the theorem follows from Propositions \ref{LCPclassic} and \ref{ekviv}.
\end{proof}

\begin{lemma}\label{posdiag}
If for $A \in \glm$ there exists $x\in \R^m$ with $\lng Ax,x\rng >0,$
then $\dy (A)$ contains a matrix with at least one positive element in its diagonal.
\end{lemma}
\begin{proof}
Let $S(A)$ be the symmetrizant of $A$. Then $\lng S(A)x,x\rng >0$.

Consider the spectral decomposition
$$S(A)=ODO\tp $$
of $S(A)$. Then,
$$\lng S(A)x,x\rng =\lng ODO\tp x,x\rng = \lng DO\tp x,O\tp x\rng >0$$
Hence,
the diagonal matrix $D$ must contain positive elements,
since $\lng Dy,y\rng= \sum_i d_{ii}y_i^2 >0$ with $y=O\tp x$.

Now, 
$$ D= \frac{O^{-1}A(O\tp)^{-1}+O^{-1}A\tp(O\tp)^{-1}}{2}=\frac{O^{-1}A(O\tp)^{-1}+(O^{-1}A(O\tp)^{-1})\tp}{2},$$  
hence the diagonal of $O^{-1}A(O\tp)^{-1}$ coincides with the diagonal of $D$ and
hence it must contain positive elements.
\end{proof}
 
The matrix $A=(a_{ij})_{i,j=1,...,m}\in \R^{m\times m}$ is called \emph{positive},
if $a_{ij}>0$, $\forall i,j$.

\begin{lemma}\label{strpos}
For the matrix $A\in \R^{m\times m}$ the following two assertions
are equivalent:
\begin{enumerate}
\item $\dy (A)$ contains a matrix with a positive element on its diagonal,
\item $\dy (A)$ contains a positive matrix.
\end{enumerate}
\end{lemma}

\begin{proof}
(a) We first prove that if the matrix $A:=(a_{ij})_{i,j=1,....,m}\in \R^{m\times m}$ contains  a positive principal submatrix of order
$n-1<m$, then it has a conjugate containing a positive principal submatrix of order $n$.

Suppose that $A(1:n,1:n):=(a_{ij})_{i,j=1,....,n}$ has the property that $a_{ij}>0$ whenever $i,j\in\{2,\dots,n\}$
and let $A(2:n,2:n)=(a_{ij})_{i,j=2,\dots,n}$.

Denote by $I\in\R^{n\times n}$ the unit matrix and let $E_{12}\in\R^{n\times n}$ be the matrix with $1$ in the position $(i,j)=(1,2)$ and 
$0$ elsewhere. Let $L_t=I+tE_{12}$ with $t$ a real parameter.
Put \[B=L_tA(1:n,1:n)L_t\tp=(b_{ij})_{i,j=1,...,n}\] and $B(2:n,2:n)=(b_{ij})_{i,j=2,\dots,n}$.
Then, we have
\begin{equation}\label{B22}
B(2:n,2:n)=A(2:n,2:n),
\end{equation}
\begin{equation}\label{b11}
b_{11}=a_{11}+ta_{21}+ta_{12}+t^2a_{22},
\end{equation}
\begin{equation}\label{b1i}
b_{1i}=a_{1i}+ta_{2i} ,\;\;i\geq 2,
\end{equation}
\begin{equation}\label{bi1}
b_{i1}=a_{i1}+ta_{i2} ,\;\;i\geq 2.
\end{equation}

From (\ref{B22}), (\ref{b11}), (\ref{b1i}) and (\ref{bi1}),
it follows that for $t>0$ large enough we will have $b_{ij}>0,\;i,j=1,...,n$.

(b) Applying the procedure from (a), the element $a_{ii}>0$ on the diagonal of $A$
can be augmented to obtain in $\dy (A)$ a matrix with positive principal minor of order $2$,
then a matrix with positive principal minor of order $3$ in $\dy (A)$, and so an, to obtain a positive
matrix in $\dy (A)$.
\end{proof}

\begin{theorem}\label{nempd}
If for $A\in \glm$ there exists an $x\in \R^m$ with $\lng Ax,x\rng >0$
and an $y\in \R^m\setminus \{0\}$ with $\lng Ay,y\rng \leq 0$, then
$\dy (A)$ contains non-\pf -matrices and \pq-matrices as well.
\end{theorem}

\begin{proof}
By Theorem \ref{eqn} $\dy (A)$ must contain non-\pf-matrices.

By Lemma \ref{posdiag}, $\dy (A)$ must contain a matrix $B$ with at least one positive element on
its diagonal. Then, by Lemma \ref{strpos}, $\dy (B)=\dy (A)$ must contain a positive matrix.
By Theorem 3.8.5 in \cite{MR3396730}, it follows that a such matrix is a \pq-matrix.
\end{proof}

\begin{corollary}\label{vegso}
Suppose that $A\in \glm$.
Then, exactly one of the following alternatives hold:
\begin{enumerate}
\item $\lng Ax,x\rng >0$, $\forall x\in \R^m\setminus \{0\}$ $\iff$
each member of $\dy (A)$ is a \pf-matrix
$\iff$
each member of $\dy (A)$ is a \pq-matrix
$\iff$
each member of $\dy (A)$ is a \pb-matrix.
\item If $\lng Ax,x\rng \leq 0$, $\forall x\in \R^m$, then no matrix in
$\dy (A)$ can have the \pf-property.
\item If for some non-zero elements $x$ and $y$ in $\R^m$ one has
$\lng Ax,x\rng >0$ and $\lng Ay,y\rng \leq 0$, then $\dy (A)$ contains
non-\pf-matrices and \pq-matrices as well.
\end{enumerate}
\end{corollary}

\begin{proof}
	Only item 2 needs proof. Let $A\in \glm$ a matrix with $\lng Ax,x\rng \leq 0$, $\forall x\in \R^m$. Suppose to the contrary that there is a matrix 
	$L\tp AL\in\dy (A)$ which has the \pf-property. Let $q\in\R^m$ be a vector with all components negative. Then, there exists $x\in\R^m_+$ such that
	\begin{equation}\label{ef}
		L\tp ALx+q\in\R^m_+.
	\end{equation}
	Thus, 
	\begin{equation}\label{epsp}
		0\le \lng x,L\tp ALx+q\rng=\lng Lx,ALx\rng+\lng x,q\rng.
	\end{equation}
	If $x\ne 0$, then the right hand side of equation \eqref{epsp} is negative, which is a contradiction. Hence, $x=0$. Then, \eqref{ef}
	implies $q\in\R^m_+$, which contradicts the choice of $q$. In conclusion, no matrix in $\dy (A)$ can have the \pf-property.
\end{proof}

The alternatives listed in Corollary \ref{vegso} can be formulated in complementarity terms too:

\begin{corollary}\label{vegso2}
Suppose that $A\in \glm$.
Then exactly one of the following alternatives hold :
\begin{enumerate}
\item $A$ possesses the $K$-\pf- property for any
simplicial cone $K$ $\iff$
$A$ possesses the $K$-\pq- property for any
simplicial cone $K$ $\iff$
$A$ possesses the $K$-\pb- property for any
simplicial cone $K$.
\item There is no simplicial cone $K$ for which $A$ possesses the $K$-\pf-property.
\item There exists a simplicial cone $K$ for which $A$ does not have
the $K$-\pf-property and there exists a simpicial cone $L$ for which
$A$ possesses the $L$-\pq-property.
\end{enumerate}
\end{corollary}

\bibliographystyle{ieeetr}
\bibliography{orbit}

\end{document}